\documentclass[12pt,letterpaper]{amsart}
\usepackage{graphicx}
\usepackage{epsfig}
\usepackage{amsmath, amssymb, amsthm, mathrsfs, verbatim}
\usepackage[left=1.4in,top=1in,right=1.4in]{geometry}
\usepackage{setspace}

\newcommand{\R}{\ensuremath{\mathbb{R}} }

\newcommand{\D}{\ensuremath{\mathscr{D}} }
\newcommand{\F}{\ensuremath{\mathscr{F}} }

\newcommand{\flr}[2]{\lfloor \frac{#1}{#2}\rfloor} 

\newtheorem*{Main Result}{Main Result}
\newtheorem{theorem}{Theorem}
\newtheorem{lemma}[theorem]{Lemma}

\begin{document}




\title{A Tree Sperner Lemma}

%
%

\author{Andrew Niedermaier}
\address{Jane Street Capital \\
1 New York Plaza, Floor 33\\
New York, NY 10004
}
\email{agnieder@gmail.com}

\author{Douglas Rizzolo}
\address{
Department of Mathematics\\
University of Washington\\
Seattle, WA 98195}
\email{drizzolo@uw.edu}
 
\author[Francis Edward Su]{Francis Edward Su$^*$}
\address{Department of Mathematics\\ Harvey Mudd College\\ Claremont,
  CA  91711}
\email{su@math.hmc.edu}

\thanks{$^*$Corresponding author.}
\thanks{
Rizzolo acknowledges partial support by NSF grant DMS-1204840 and NSF Graduate Research Fellowship grant DGE-1106400.
Su acknowleges
partial support by NSF Grants DMS-0301129, DMS-0701308, DMS-1002938.
}



\maketitle

\begin{abstract}
In this paper we prove a combinatorial theorem for finite labellings of trees,
and show that it is equivalent to a theorem for finite covers of metric trees 
and a fixed point theorem on metric trees.  We trace how these connections mimic the equivalence of the Brouwer fixed point theorem with the classical KKM lemma and Sperner's lemma.  We also draw connections to a KKM-type theorem about infinite covers of metric trees and fixed point theorems for non-compact metric trees.  Finally, we develop a new KKM-type theorem for cycles, and discuss interesting social consequences, including an application in voting theory.
\end{abstract}

\section{Introduction}
The Brouwer fixed point theorem is a celebrated topological result 
that says every continuous map of an $n$-ball to itself has a fixed point.
It is known \cite{KKMa29, Yose74} to be equivalent to a set-covering result known as the KKM lemma, and a combinatorial result known as Sperner's lemma.  Aside from their intrinsic interest, these equivalent formulations have led to simpler methods for proving the Brouwer result as well as practical algorithms for finding fixed points of highly non-linear functions (see e.g., \cite{todd}).  Other topological theorems such as the Borsuk-Ulam theorem also admit similar set-covering and combinatorial formulations \cite{NySu13}.

Like the $n$-ball, a finite tree as a topological space also has the \emph{fixed point property}: every continuous map of a tree to itself has a fixed point.  A primary goal of this paper is to explore a combinatorial analogue of the fixed point property for trees and draw a connection to a set-covering analogue, akin to the Sperner and KKM lemma analogues of the Brouwer theorem.  

Our new combinatorial analogue is Theorem \ref{thm:tree-sperner} which we call the Tree Sperner Lemma because of its similarities with Sperner's Lemma.  We show that this is equivalent to a new Tree Fixed Vertex-Edge Theorem (Theorem \ref{thm:tree-discrete}) involving functions defined only on the vertex set of a combinatorial tree.  These are proved in 
Section \ref{sc Main} and are easy to establish, but lead to simpler proofs of some known results about metric trees: a Tree KKM Theorem (Theorem \ref{thm:tree-kkm}) in Section \ref{sc KKM} and a Tree Fixed Point Theorem (Theorem \ref{thm:tree-fixed}) in Section \ref{sc finite trees}.  Theorem \ref{thm:equiv} shows that these four results are equivalent.

Because finite trees are compact and acyclic, the fixed point property for trees follows from the Lefschetz fixed point theorem just like the Brouwer theorem does.  However, this approach gives little insight into the location of a fixed point, and the Lefschetz theorem is not easy to prove.  By contrast, our Tree Sperner Lemma (i) gives an accessible proof of the fixed point result for trees, (ii) suggests where the corresponding fixed point is and a constructive procedure for finding it, and (iii) is of intrinsic interest due to its similarity with Sperner's lemma.

Moreover, the Tree Sperner Lemma applies more generally to infinite trees, as long as the label set is finite.
So we also show in Section \ref{sc infinite trees} how it can be used to prove a known KKM-type result about infinite covers of trees (Theorem \ref{thm:tree-kkm-infinite}), as well as a fixed point result for compact maps of infinite trees (Theorem \ref{thm:tree-fixed-infinite}).



Finally, we use the Tree KKM Theorem in Section \ref{sc cycles} to prove a new KKM-type result for covers of cycles (Theorem \ref{thm:cycle-kkm}).  Along the way we also consider interesting social interpretations of our results, including applications to voting theory.


\section{A Tree Sperner Lemma}
\label{sc Main}

The usual Sperner's lemma starts with a triangulated $n$-simplex $\Delta$ whose vertices have a \emph{Sperner labeling}: 
\begin{itemize}
\item  each main vertex of $\Delta$ has a distinct label (chosen from $n+1$ labels), and
\item  each vertex $v$ of the triangulation is assigned a label of one of the main vertices spanning the minimal face of $\Delta$ that $v$ is on.
\end{itemize}
For instance, if $v$ is on the edge of $\Delta$ spanned by $a$ and $b$, then $v$ must be labelled either $a$ or $b$.

For such a labelling, Sperner's lemma asserts that there must be a \emph{fully-labelled} simplex, i.e., one with all $n+1$ labels.
In Theorem \ref{thm:tree-sperner}, we develop an analogous combinatorial theorem for \emph{proper} labellings of $n$-vertex trees by $n$ labels that will assert the existence of an edge with all $n$ labels.  

For this result, we view trees
as combinatorial (i.e., connected acyclic graphs specified by vertices and edges), although 
in subsequent sections we shall consider the implications of our result for metric trees 
(tree-like metric spaces).

Let $T=(V,E)$ be a tree with vertex set $V$ and edge set $E$.  
To avoid trivialities, we assume $V$ has at least two vertices.
If $V$ is finite, we say $T$ is a {\em finite} tree; otherwise $T$ is {\em infinite}.  Note that even for an infinite tree, 
between any two vertices $u$ and $w$, there is a finite chain of edges that connect $u$ and $w$ and this path of edges is unique.

If $v$ is a vertex, then let $T \setminus v$ denote the graph that
results from removing $v$ from $V$ and all the edges incident to $v$
from $E$.  This new graph may have several connected components.   Similarly, for an element $e=\{v_1,v_2\}\in E$, we let $T\setminus e$ denote the graph that results from removing $e$ from $E$.   We note that, since $T$ is a tree, $T\setminus e$ has exactly two components.

\begin{figure}[htb]
  \begin{center}
    \scalebox{.7}{\includegraphics{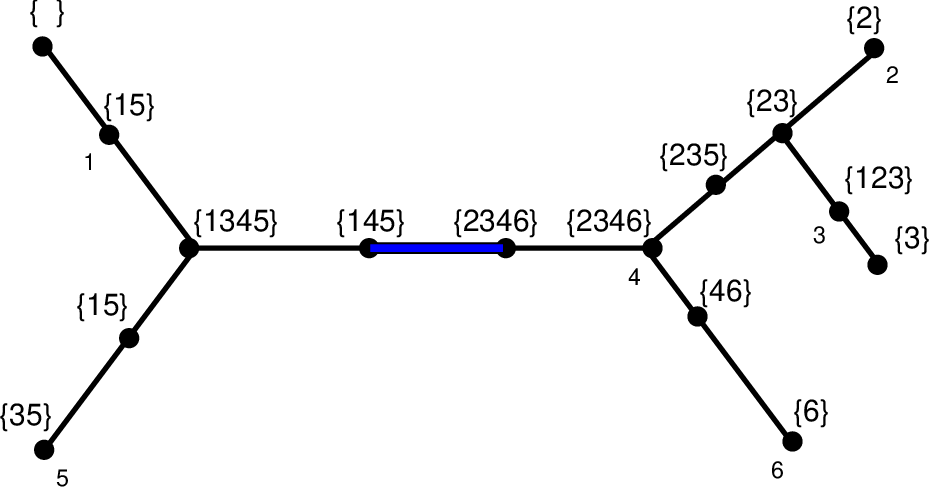}}
  \end{center}
  \caption{A proper labelling.  The non-bracketed numbers mark the vertices that are labels and the bracketed numbers denote the labels of each vertex.  The thickened edge is a fully-labelled edge.}
  \label{fig-fully-labelled}
\end{figure}

Let $A$ be a subset of $V$ which we call the {\em labels}; every vertex of $V$ will be assigned
a collection of labels by a labelling function $\ell$.  Let $2^A$ denote the power set of $A$, i.e., the set of all subsets of $A$.

We call a labelling $\ell:V \rightarrow 2^A$ {\em proper} if:
\begin{itemize}
\item  for each $a \in A$, $\ell(a)$ contains $a$, and 
\item  if $v \in V$ lies on the unique path between $a, b \in A$, then $\ell(v)$ contains either $a$ or $b$.
\end{itemize}
Note how these conditions are analogous to those of the usual Sperner's lemma.  It is easy to verify that they are equivalent to this one condition:
\begin{itemize}
\item  for each $v \in V$, if the set $A \setminus \ell(v)$ is nonempty, then it consists of vertices that all lie in one component of $T \setminus v$.
\end{itemize}
We may think of this condition as saying that $\ell(v)$ is missing labels from at most one component of $T\setminus v$.
Similarly, if $e$ is an edge incident to $v$, then $\ell(v)$ is missing labels from at most one component of $T\setminus e$.
Figure \ref{fig-fully-labelled} shows an example of a tree with a proper labelling.

Let $e$ be an edge with endpoints $x,y$.
We call that edge \emph{fully-labelled} if $\ell(x)\cup \ell(y) = A$, i.e., if the edge 
contains all labels.  The proper labelling in Figure \ref{fig-fully-labelled} has a
fully-labelled edge. This exemplifies our theorem, which may be viewed as an analogue of Sperner's lemma for trees:

\begin{theorem}[Tree Sperner Lemma]
\label{thm:tree-sperner}
Let $T=(V,E)$ be a tree, let $A$ be a finite subset of $V$, and let $\ell:V \rightarrow 2^A$ 
be a proper labelling.   Then $T$ contains a fully-labelled edge.
\end{theorem}

The finiteness of $A$ is essential, as one may see by considering the {\em integer tree}: whose vertices are the integers and whose edges connect successive integers (see Figure \ref{fig-integer-tree}).  Now let $A=V$ and label each vertex 
$n$ by all the integers less than or equal to $n$.  This labelling is proper, but it does not have a fully-labelled edge.

\begin{proof}
It suffices to consider the case where $V$ is also finite, for if not, we may restrict our attention to the finite subtree of $T$ spanned by the vertices of $A$,
noting that any fully-labelled edge in this subtree is fully-labelled in the whole tree.  
  
  If there exists $v$ in $V$ with $\ell(v)=A$, then desired conclusion holds: any edge
  containing $v$ is fully-labelled.  Otherwise, we can construct a ``successor'' function $S: V \rightarrow V$.
  For each $v$ in $V$, the labels $A \setminus \ell(v)$ come from 
  exactly one component of $T \setminus v$.  In that component, let $S(v)$ be the unique 
  vertex that shares an edge with $v$ in $T$.

  Consider the sequence $\{v_n\}$ defined by choosing some $v_1 \in A$ 
  and letting $v_k = S(v_{k-1})$ for $k>1$.   Since $V$ is finite, this sequence must be 
  eventually periodic.  Since $T$ is a tree, this period cannot be of length greater
  than two.  Thus $S(x)=y$ and $S(y)=x$ for some $x,y$ that are
  endpoints of an edge $e$.  
  
  We claim $e$ is fully-labelled.  
  Since $S(y)=x$, the labels $A \setminus \ell(y)$ are in the component of $T\setminus y$ containing $x$.  
  Since $S(x)=y$, the labels $A \setminus \ell(x)$ are in the component of $T\setminus x$ containing $y$.  
  Then $A \setminus \ell(y)$ and $A \setminus \ell(x)$ are disjoint because they are in different components of 
  $T\setminus e$.  Hence $\ell(x)\cup \ell(y) = A$ so that $e$ is a fully-labelled edge. 
\end{proof} 

Note that the above successor function yields a systematic method for
locating a fully-labelled edge, by iterating $S$ until one reaches a vertex for
which $S$ is not defined (and therefore has all labels), or until the
sequence repeats.  This avoids exhaustively checking vertices, which can be problematic if $V$ is infinite.

Theorem \ref{thm:tree-sperner} yields an interesting corollary for functions 
whose domain and range are vertices of $T$.
\begin{theorem}[Tree Fixed Vertex-Edge Theorem]
\label{thm:tree-discrete}
Suppose $T=(V,E)$ is a tree and $f:V\to V$ is a function on vertices with finite range.  
Then either $f$ fixes some vertex, or 
there is an edge $e\in E$ with endpoints $x$ and $y$ such that $e$ is on the path from $f(x)$ to $f(y)$.
\end{theorem}

Theorem \ref{thm:tree-discrete} may be viewed as a kind of ``discrete'' fixed point theorem for trees, 
because it says that either there is a fixed vertex, or some edge must be covered by the path between the images of its endpoints.  Compare it to the continuous version we prove later in Theorem \ref{thm:tree-fixed}.  

\begin{proof}
Let $A$ be the range of $f$.  Suppose $f$ does not fix any vertex.
Consider a labelling $\ell: V\to 2^A$ such that for each vertex $v$, 
$\ell(v)$ is the set of all vertices of $A$ that are not in the component of
$T\setminus v$ containing $f(v)$ (including $v$ if $v \in A$).
The labelling is by definition proper.
Theorem \ref{thm:tree-sperner} implies that there is a fully-labelled edge
$e$ with endpoints $x$ and $y$.

Since by definition $f(y) \notin \ell(y)$, we must have $f(y) \in \ell(x)$.  
Thus $f(y)$ is not in the component of $T \setminus x$ containing $f(x)$, 
so $x$ must be on the path between $f(x)$ and $f(y)$.
Similarly, $f(x) \notin \ell(x)$ implies $f(x) \in \ell(y)$.
Thus $f(x)$ is not in the component of $T \setminus y$ containing $f(y)$, 
so $y$ must be on the path between $f(x)$ and $f(y)$.
Thus $e$ must be on the path between $f(x)$ and $f(y)$, as desired.
\end{proof}

Moreover, 
\begin{theorem}
\label{thm:first-two}
Theorem \ref{thm:tree-discrete} is equivalent to Theorem \ref{thm:tree-sperner}.
\end{theorem}

\begin{proof}
Having already shown 
Theorem \ref{thm:tree-sperner} implies Theorem \ref{thm:tree-discrete}, we now show the converse.
Suppose that $\ell: V \rightarrow 2^A$ is a proper labelling of $T$, with $A$ finite.
If there is a vertex $v$ such that $\ell(v)=A$, then any edge containing $v$ is fully-labelled.
Otherwise, for each vertex $v$ define $f(v)$ to be a label from $A$ which is not in $\ell(v)$ 
(if there are several options, choose one).
Clearly $f:V \rightarrow A$ has no fixed points, 
so Theorem \ref{thm:tree-discrete} implies that there is an edge $e=\{x,y\}$ such that $e$ is on the path from $f(x)$ to $f(y)$.

We claim that $e$ is fully-labelled.  
Since $f(x)$ and $f(y)$ are in different components of $T \setminus e$ and $\ell$ is proper,
then the definition of $f$ shows that $\ell(x)$ and $\ell(y)$ are missing labels from different components of $T\setminus e$.  
So there are no
vertices of $A$ that are missing from both $\ell(x)$ and $\ell(y)$, i.e., 
$\ell(x)\cup \ell(y) = A$ and $e$ is a fully-labelled edge.
\end{proof}

\section{Metric Trees and Segmentations}
\label{sec:metric-trees}

Theorems \ref{thm:tree-sperner} and \ref{thm:tree-discrete} have several applications to metric trees, which are essentially combinatorial trees realized as metric spaces by replacing edges with line segments isometric to a compact interval of $\R$.  We make precise in this section what we mean by {\em metric tree} as well as the concept of a subdivision called a {\em segmentation}, but there are no surprises here, so this brief section may be skimmed if desired. 

A {\em metric tree} is a triple $T=(V,E,X)$, with a vertex set $V$ (that may be finite or infinite), an edge set $E$, and {\em underlying metric space} $X$.  Here, $(V,E)$ specifies a combinatorial tree and the metric space $X$ is obtained from 
$(V,E)$ by realizing every edge $e$ as an isometric copy of some closed interval $[0,L_{e}]$ and gluing the realized edges according to the instructions in $(V,E)$.  The number $L_{e}$ is called the \emph{length} of edge $e$.

%

Since $(V,E)$ has no cycles, between any two points $x,y$ in $X$ there is a 
unique non-self-intersecting path between $x$ and $y$. 
There is a natural metric on $X$: let $d(x,y)$ be the length of this path between $x$ and $y$,
i.e., the sum of the lengths of the edges (or partial edges) along this unique path.
It will be useful to note if $z$ is on the path between $x$ and $y$, then the triangle inequality becomes an equality: $d(x,y)=d(x,z)+d(z,y)$.

We remark that our definition of metric tree differs from others in the literature because it retains the combinatorial structure specified by $V$ and $E$.

Given a tree $T=(V,E,X)$ it will be convenient at times to consider a {\em segmentation} of $T$, which is another metric tree obtained from $T$ by finite subdivision of its edges.  In particular, 
$T'=(V',E',X)$ is a {\em segmentation} of $T=(V,E,X)$ if: (1) $V'=V \cup V^*$ where
$V^*$ is a collection of points $\{ v_\alpha \}$ from $X$ so that at most a finite number of the $v_\alpha$ come from a realized edge $e_X$, $e \in E$, and (2) $E'$ is the collection of edges obtained from $E$ in the natural way (by deleting edges in which elements of $V^*$ appear and including edges of the implied subdivision along that edge).  Note that the metric spaces for $T'$ and $T$ are the same, so the set of continuous functions on $T$ and $T'$ are the same.


The {\em size} of a segmentation $T'=(V',E',X)$  is defined by $size(T') = \sup_{e\in E'} L_{e}$, and bounds the size of the longest edge.  Clearly every tree has an arbitrarily small segmentation.

Note also that every point in a metric tree $T$ that is not a leaf (a vertex of degree 1) is a {\em cut point}: its removal ``cuts'' $T$ into more than one path-connected component.

In what follows, all trees $T$ are metric trees.

\begin{figure}[htb]
  \begin{center}
    \scalebox{.75}{\includegraphics{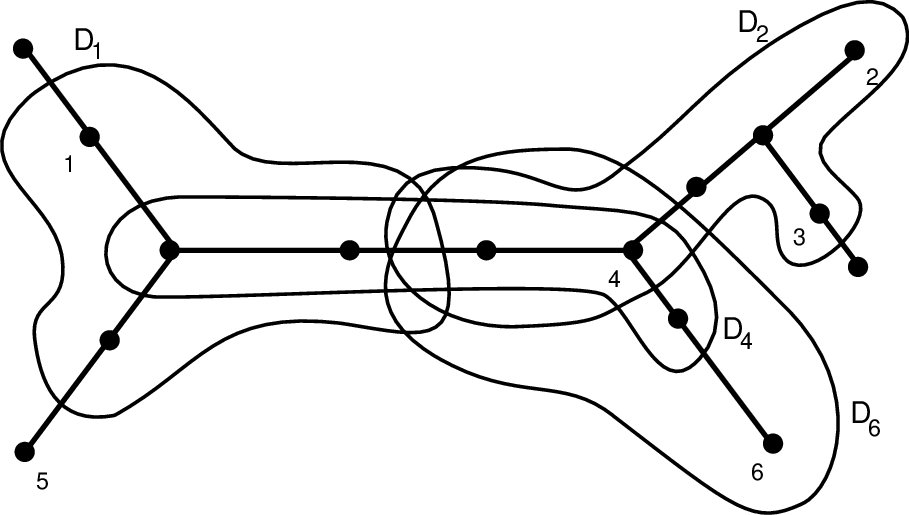}}
  \end{center}
  \caption{A KKM cover of a tree relative to the vertices $\{1,2,4,6\}$.  The sets of a KKM cover do not have to be connected sets (although they are in this diagram).}
  \label{fig-kkmcover}
\end{figure}

\section{KKM Covers of Trees}
\label{sc KKM}

The usual KKM lemma \cite{KKMa29} starts by considering an $n$-simplex $\Delta$ with a \emph{KKM cover} by $n+1$ closed sets $\{C_{i}\}$: these have the properties that
\begin{itemize}
\item each main vertex $v_{i}$ has an associated closed set $C_{i}$ that contains it
\item the face of $\Delta$ spanned by $\{v_{i_{1}}, \dots, v_{i_{k}}\}$ is covered by  $\{C_{i_{1}}, \dots, C_{i_{k}}\}$.
\end{itemize}
Under these conditions, the KKM lemma says there exists a point in the intersection of all $n+1$ sets.  

Just as the traditional Sperner's Lemma implies the classical KKM theorem for covers of simplices by closed sets,
Theorem \ref{thm:tree-sperner} implies a KKM theorem for metric trees.

Let $T=(V,E,X)$ be a tree, and $A$ a subset of points in $X$.  
Call a family of closed sets $\D = \{D_a : a \in A \}$ a 
\textit{KKM cover of $T$ relative to $A$} if:
\begin{itemize}
\item each $a \in D_a$, and 
\item for any two points $a$ and $b$ in $A$, the path between them is contained in $D_a \cup D_b$.
\end{itemize}
We call the last condition the \textit{path-covering property} of $\D$; it is analogous to the face-covering property of KKM covers of simplices.  
See Figure \ref{fig-kkmcover}.
  
If $A=V$, we may simply say $\D$ is a \textit{KKM cover} of $T$, and the the sets of $\D$ will cover the underlying metric space $X$.  
(Note that if $A \neq V$, then a KKM cover of $T$ relative to $A$ may not cover $X$, but it will cover the subtree spanned by $A$.)

Each KKM cover $\D$ relative to a (finite or infinite) subset of vertices defines a \emph{membership labeling} 
by assigning to a vertex $v$ the indices of 
all the sets of $\D$ that contain it.  (The label set may be empty if $v$ is not covered by the sets of $\D$.)  
This labeling is proper:

\begin{lemma} 
\label{lem:proper}
Let $T=(V,E,X)$ be a tree, let $A \subset V$ be a subset of vertices, and 
let $\D= \{D_a : a \in A\}$ be a KKM cover of $T$ relative to $A$.  
Then the labelling $\ell:V \rightarrow 2^A$ defined by 
$\ell(v) = \{a : v \in D_a\}$ is proper.
\end{lemma}

As an example, in Figure \ref{fig-kkmcover}, all vertices in the set $D_1$ will have $1$ in their label set.  Similarly, because vertex 4 is in sets $D_2$ and $D_4$ and $D_6$, $\ell(4)$ will contain $1$, $2$, and $4$.  The leaf at top left will have an empty label set.  The reader may notice that Figure \ref{fig-fully-labelled} gives the membership labelling for the KKM cover in Figure \ref{fig-kkmcover} 
if labels $3$ and $5$ were removed from every label set.

\begin{proof}
Fix a vertex $v$ in $V$, and suppose there were two vertices $a$ and $b$ in $A$ that are not in $\ell(v)$.
Then by definition $v$ is not in $D_a$ nor in $D_b$.  Since $a \in D_{a}$ and $b \in D_{b}$, then neither $a$ nor $b$ can be the vertex $v$.  

If $a$ and $b$ were in different components of $T\setminus \{v\}$, 
then $v$ must lie on the unique path between $a$ and $b$ in $T$ and the path-covering property of $\D$ would imply $v \in D_a \cup D_b$, 
a contradiction.  Therefore any vertices not in $\ell(v)$ must lie in one component of $T\setminus \{v\}$, as desired.
\end{proof}

We now use our Tree Sperner Lemma to prove a known theorem about finite KKM covers of trees.  Although it follows from the results in \cite{Berg05} and \cite{Kham96}, our proof is more elementary.

\begin{theorem}[Tree KKM Theorem]
\label{thm:tree-kkm}
Let $T=(V,E,X)$ be a metric tree, $A$ a finite subset of points of $X$,
and suppose $\D= \{D_a : a \in A\}$ is a KKM cover of $T$ relative to $A$.  Then
\[\bigcap_{a \in A} D_a \neq \emptyset.\]
\end{theorem}

\begin{proof}
We may as well assume that $V$ is finite, 
for otherwise we may restrict our attention to the subtree $K$ 
spanned by a finite set of edges that contain $A$, which contains a finite number of vertices.
Any KKM cover of $T$ relative to $A$ will also restrict to a KKM cover of $K$ relative to $A$, 
and a nonempty intersection of the KKM cover of $K$ would imply a nonempty intersection of the KKM cover of $T$.

Suppose, by way of contradiction, that the intersection
$\cap_{a \in A} D_a$ were empty.
Then the set of complements
$\mathscr{C} = \{T\setminus D_a : a \in A \}$ 
is an open cover of $T$.
Since $V$ is finite, $X$ is compact and 
this cover has a Lebesgue number $\delta$.  
Let $T'=(V',E',X)$ be a segmentation of $T$ with $size(T') < \delta$ chosen so that $A$ is a subset of the vertices of $V'$.   

Consider the membership labelling $\ell:V'\rightarrow 2^A$ defined by 
$\ell(v') = \{a : v' \in D_a\}$.   
Lemma \ref{lem:proper} shows that $\ell$ is a proper labelling.   
By the Tree Sperner Lemma,  
there exists  a fully-labelled edge $e \in T'$ with
endpoints $x$ and $y$ such that $\ell(x)\cup \ell(y) = A$.   
Thus, for all $a \in A$, either $x \in D_a$ or $y \in D_a$ (or both).
However, since $size(T') < \delta$, the Lebesgue number property guarantees that 
$e \subseteq T\setminus D_a$ for
some $a$, implying that $e \cap  D_a = \emptyset$, a
contradicting that $e$ was fully-labelled.   Therefore, we conclude that $\cap_{a \in A} D_a
\neq \emptyset$.
\end{proof}

Note also that the sets of a KKM cover do not have to be connected (though they are in Figure \ref{fig-kkmcover}.  However, if a tree is covered by sets that are connected as well as pairwise intersecting, then it is a KKM cover!

\begin{theorem}[Tree KKM for Connected Sets]
\label{cor:tree-kkm-connected}
Let $\D= \{D_1, D_2, ..., D_k\}$ be a finite collection of closed, connected sets that cover a metric tree $T=(V,E,X)$ such that each pair $D_i \cap D_j$ is nonempty.  Then there is a point $x$ in all the sets of $\D$.
\end{theorem}

\begin{proof}
Choose points $a_i \in D_i$ for each $i$, and put them in a set $A$.  To show $\D$ is a KKM cover of $T$ relative to $A$, it remains to show the path-covering property.

If for some pair $a_i$ and $a_j$ in $A$, the path-covering property did not hold, then the path between $a_i$ and $a_j$ would contain a point $y$ that is not covered by $D_i \cup D_j$.  Then $X \setminus \{y\}$ would have two connected components that would separate $a_i$ from $a_j$.  Then $D_{i}$ and $D_{j}$ must lie in different components because each is connected.  But then they could not be pairwise intersecting, a contradiction.
\end{proof}


We indicate some implications of Theorem \ref{thm:tree-kkm} below, and note that the proof of Theorem \ref{thm:tree-sperner} would suggest associated constructive algorithms.

\subsection*{Pizza Delivery.}
Suppose you are starting a pizza delivery business and you desire a good location for your store.  Your city has
several neighborhoods connected by a tree of roads.  
For each neighborhood $i$, there is a ``deliverability'' set $D_i$: the set of all
locations on the tree with an acceptable commute to neighborhood $i$.
These sets are naturally closed and connected, as in Figure \ref{fig-kkmcover}.
Then the Tree KKM Theorem for Connected Sets (Theorem \ref{cor:tree-kkm-connected}) says that if for every pair of neighborhoods
$i$ and $j$ there is a common acceptable location to place your store,
then there will be a location with an acceptable commute to all
neighborhoods.

\subsection*{Grand Central Station.}
Suppose several cities are connected by a tree of train tracks.  Each
city has its own train authority, and suppose it is possible to get from city
$i$ to city $j$ using only those cities' trains (possibly switching
several times).  
Then Theorem \ref{thm:tree-kkm} shows that there must be a location through which trains
from all cities must pass, i.e., there is a location where one could
place a Grand Central Station.

Note that something further is true if we make some intuitive
assumptions about the structure of the tree of train tracks.
It seems reasonable to suppose that each vertex of this tree 
is a station and that trains only change directions at stations.   
With these assumptions the following result
becomes apparent: there must already be a station at which trains from all
of the cities stop, i.e., a Grand Central Station already exists.
If the point guaranteed by Theorem \ref{thm:tree-kkm} is not a station,
then trains from each city also must pass through the two nearest stations on either side of this
point (because trains only change directions at stations).

\section{A Fixed Point Theorem for Finite Trees}
\label{sc finite trees}
Just as the KKM Theorem and Sperner's Lemma imply Brouwer's Theorem on
simplices, we can use our previous theorems to prove a classical 
fixed point theorem for trees.

\begin{theorem} (The Tree Fixed Point Theorem)
\label{thm:tree-fixed}
Let $T=(V,E,X)$ be a metric tree, $V$ be finite, and $f:T\rightarrow T$ be a continuous function.  Then $f$ has a fixed point.
\end{theorem}

As already noted, this result follows from the Lefschetz fixed point theorem, 
but the proofs we give below have an advantage in being elementary.  
The first is a non-constructive proof using the Tree KKM Theorem, 
the second is a constructive approach using the Tree Sperner Lemma. 

We shall often appeal to a special family of closed sets 
associated to any continuous self-map on a tree.
Given $T=(V,E,X)$ be a metric tree, and $A \subset V$ a subset of vertices,
and $f:T\rightarrow T$ be a continuous function, let  
$\D_{f,A} = \{D_a: a \in A\}$ be the family of sets defined by
\begin{equation*}
D_a = \{x\in T : d(x,a) \leq d(f(x),a)\}.
\end{equation*}
To interpret, $D_a$ contains the set of all points that stay the same distance or move away from $a$.
We now show this family $\D_{f,A}$ is a KKM cover, so we shall refer to it as the 
{\em move-away KKM cover associated to $f$ and $A$}.

\begin{theorem}
\label{thm:fixedpt-kkmcover}
The family $\D_{f,A}$ defined above is a KKM cover of $T$ relative to $A$.
\end{theorem}

\begin{proof}
It is apparent from the definition that $a \in D_a$ for each $a \in A$.

We now show that each $D_a$ is closed.  
Let $\{x^k\}$ be a sequence in $D_a$ that converges to
  $x$ in $T$.  Since $x^k \in D_{a}$, we have $d(x^{k},a) \leq d(f(x^{k}),a)$.
  Since $f$ is continuous, we see that
  $f(x^k) \rightarrow f(x)$.   
  
Let $\epsilon >0$ be given.   Choose
  $N$ such that $k\geq N$ implies that $d(x^k,x) < \epsilon$ and
  $d(f(x^k),f(x))<\epsilon$.   Then
 \[\begin{split} d(x,a) & \leq d(x,x^k)+d(x^k,a) \\
 & \leq d(x,x^k) + d(f(x^k),a) \\
 &\leq d(x,x^k)+d(f(x^k),f(x))+d(f(x),a)\\
 & < 2\epsilon+d(f(x),a).
 \end{split}\]
  Since this is true for all $\epsilon>0$ we conclude that $d(x,a)\leq
  d(f(x),a)$, so $x\in D_a$.   Therefore $D_a$ is
  closed.

Finally, we show the path-covering property of $\D_{f,A}$.
If $a$ and $b$ are in $A$, consider $x$ on the path between them.
If $f(x)=x$, then $x$ doesn't move, so $x\in D_a \cup D_b$ as desired.
Otherwise, suppose that $f(x)$ is not in
  the same component of $T\setminus \{x\}$ as $a$.
  Then the path from $a$ to $f(x)$ and must contain $x$.   
  It follows that the path from $a$ to $x$ is contained in the path from $a$ to
  $f(x)$, so we conclude that $d(a,x)\leq d(a,f(x))$, hence $x\in D_a$.   
  By similar argument,
  if $f(x)$ is not in the same component of $T\setminus \{x\}$ as
  $b$, then $x\in D_b$.  Either way, $x\in D_a\cup D_b$.
\end{proof}

We can now give a quick proof of the Tree Fixed Point Theorem:

\begin{proof}[First Proof of Theorem \ref{thm:tree-fixed}.]
  Consider the move-away KKM cover $\D_{f,V}$ associated to $f$ and $V$.
  Since $V$ is finite, the Tree KKM Theorem shows there exists a point $x$ in all sets of $\D_{f,V}$.
  By definition this means $d(x,v) \leq d(f(x),v)$ for all $v \in V$.
   
  Suppose that $x\neq f(x)$.   If $f(x)$ were a vertex $v$, then this would contradict $x \in D_{v}$.
  So $f(x)$ is not a vertex and is thus a cut point.
  Choose $w \in V$ such that $w$ and $x$ are
  in different components of $T\setminus \{f(x)\}$.   
  Then the path from $w$ to $f(x)$ is 
  contained in the path from $w$ to $x$.   Since
  $f(x)\neq x$, this implies that $d(x,w) > d(f(x),w)$, so that $x
  \notin D_w$, a contradiction.   Hence $f(x)=x$.
\end{proof}

A second proof of the Tree Fixed Point Theorem relies on the following rather standard lemma that we include for completeness.  An $\epsilon$-fixed point for $f$ is an approximate fixed point $x$ such that $d(x,f(x))<\epsilon$.

\begin{lemma}[Epsilon Fixed Point Theorem] 
\label{le-epsilonfixed}
 Suppose that $K$ is a
  compact subset of the metric space $(X,d)$ and that 
  $f:K\rightarrow K$ is continuous.   If $f$ has an $\epsilon$-fixed point for every
  $\epsilon > 0$ then $f$ has a fixed point. 
\end{lemma}

\begin{proof} Let $\{a_n\}$ be a sequence of $1/n$-fixed points, 
  that is, $d(a_n,f(a_n)) < 1/n$ for all $n$.   Since $K$ is compact, $\{a_n\}$ has a convergent
  subsequence $\{a'_n\}$ converging to $x \in
  K$.   Let $\epsilon >0$.   Since $a_n'\rightarrow x$ there exists
  $N_1$ such that $n\geq N_1$ implies that $d(a'_n,x) < \epsilon /2$.
  Let $N = \max ( N_1, 2/\epsilon)$.   Then $n\geq N$ implies that
$$
d(x, f(a'_n)) \leq d(x,a'_n) + d(a'_n,f(a'_n)) < \epsilon,
$$
so that $f(a'_n) \rightarrow x$.   However, since $f$ is continuous,
we have also that $f(a'_n)\rightarrow f(x)$.   Hence $f(x)=x$, and $x$ is a desired fixed point.
\end{proof}

Now if we construct the membership labeling associated with the move-away KKM cover in the prior proof, the Tree Sperner Lemma gives a fully-labelled edge.  The next proof of the Tree Fixed Point Theorem shows that a point in this edge is an $\epsilon$-fixed point, which means that locating such an edge, as the Tree Sperner Lemma allows us to do, will allow us to find approximate fixed points.


\begin{proof}[Second Proof of Theorem \ref{thm:tree-fixed}.]
By Lemma \ref{le-epsilonfixed}, it is sufficient to show that $f$ has an $\epsilon$-fixed point for all $\epsilon >0$.   

Fix $\epsilon>0$.  Since $V$ is finite, $X$ is compact, hence $f$ is uniformly continuous.  
So there exists $\gamma>0$ such that if $d(x,y)<\gamma$ then $d(f(x),f(y)) < \epsilon/2$.   
Let $\delta = \min(\gamma, \epsilon/2)$ and let $T'=(V',E',X)$ be a segmentation of $T=(V,E,X)$ with $size(T') < \delta$.   

Let $\ell: V' \rightarrow 2^V$ be defined by $$\ell(v')=\{ v \in V : d(v',v) \leq d(f(v'),v) \}.$$
Note that this is just the membership labelling defined in Lemma \ref{lem:proper} for the move-away KKM cover $\D_{f,V}$
in the prior proof.  By the Tree Sperner Theorem, there is a fully-labelled edge $e' \in E'$ with endpoints $y$ and $z$.  We claim that $y$ is the desired $\epsilon$-fixed point.

Suppose $e'$ intersects its image $f(e')$ in some point $w$.
Then 
$$d(y,f(y)) \leq d(y,w) + d(w,f(y)) < \delta +\frac{\epsilon}{2} \leq \frac{\epsilon}{2}+\frac{\epsilon}{2} \leq \epsilon,$$
where the second inequality uses the fact that $w$ is a point in $e'$ as well as in $f(e')$.  
So in this case, $y$ is an $\epsilon$-fixed point.

Otherwise, if $e' \cap f(e')$ were empty, then since $f$ is continuous, the image $f(e')$ is connected and contained in one component of the set $X \setminus e$.  The point $f(z)$ is contained in some edge $e$ of $E$.
If $f(z) \in V$ then let $v=f(z)$, else let $v$ be the unique endpoint of $e$ that is 
in a different component of $X \setminus f(z)$ than $e'$.  
Either way, $f(z)$ is on the path from $v$ to $z$, and $f(z) \neq z$.  
So $v \notin \ell(z)$ because $f(z)$ is strictly closer to $v$ than $z$ is. 
But then $v \in \ell(y)$ because $e'$ is fully labelled, so $d(v,y) \leq d(v,f(y))$. 
Since either $f(z)=v$ or $f(z)$ separates $v$ from $y$, we see that $f(z)$ is on the path from $v$ to $y$.  Hence
$$d(v,f(z)) + d(f(z),y) = d(v,y) \leq d(v,f(y)) \leq d(v,f(z)) + d(f(z),f(y)),$$
which implies that $d(y,f(z)) \leq d(f(y),f(z))$.
But $d(f(y),f(z)) < \epsilon/2$.   So 
$$d(y,f(y)) \leq d(y,f(z)) + d(f(z),f(y)) < \frac{\epsilon}{2} + \frac{\epsilon}{2} <\epsilon.$$
Again, $y$ is an $\epsilon$-fixed point.
\end{proof}

The constructive nature of the Tree Sperner Lemma suggests a method for locating an approximate fixed point in a systematic way.  Namely, we can choose a segmentation of sufficiently small size, start at any vertex, and 
``move in the direction of the missing labels'' as suggested by the proof of Theorem \ref{thm:tree-sperner}.  This will eventually lead to a fully-labelled edge whose endpoints are approximate fixed points.

One can also use this method to converge to an actual fixed point.  In much the same way that homotopy algorithms for Sperner's Lemma can be used to ``home in'' on a fixed point by changing the size of a triangulation as one proceeds.
In the case of Sperner's lemma, one can, for example, define a triangulation on $\Delta \times \R$ that interpolates between triangulations on $\Delta$ with different mesh sizes; these are called \emph{homotopy algorithms} (e.g., see \cite{todd}, \cite{yang}).  Similarly, one can construct homotopy algorithms for trees in much the same way, by defining a triangulation on $T \times \R$.

We note that Theorem \ref{thm:tree-fixed} implies the Tree Fixed Vertex-Edge Theorem 
(Theorem \ref{thm:tree-discrete}), and thus:

\begin{theorem}
\label{thm:equiv}
The following are all equivalent:
\begin{enumerate}
\item Tree Sperner Lemma (Theorem \ref{thm:tree-sperner}),
\item Tree Fixed Vertex-Edge Theorem (Theorem \ref{thm:tree-discrete}),
\item Tree KKM Theorem  (Theorem \ref{thm:tree-kkm}), and 
\item Tree Fixed Point Theorem (Theorem \ref{thm:tree-fixed}).
\end{enumerate}
\end{theorem}

\begin{proof}
Theorem \ref{thm:first-two} showed that (2) and (1) are equivalent.  
We have also seen that (1) implies (3), and (3) implies (4), so it 
suffices to show that (4) implies (2).

Let $(V,E)$ be a tree, and let $f:V \rightarrow V$ be a function with finite range.
We construct a finite metric tree in the following way.
Let $(V',E')$ be the finite subtree of $(V,E)$ spanned by vertices in the range of $f$.  
We can realize this subtree as a metric tree $T=(V',E',X)$ by letting all edges have length $1$.
Since $(V',E')$ is finite, the underlying space $X$ is compact.

Now construct a continuous function $\widehat f:X \rightarrow X$ by extending the 
given $f$ linearly across single edges, i.e., if $x$ is a point that is some fraction of distance 
along an edge from vertex $v$ to $w$, 
then we let $\widehat f(x)$ be the point that is the same fraction of the distance along the path from $f(v)$ to $f(w)$.

By Theorem \ref{thm:tree-fixed}, there is a point $z$ such that $\widehat f (z) = z$.  If $z \in V$ then we see $f$ has a fixed vertex, as desired.  Otherwise, $z$ is on some edge $e=\{v,w\}$ and $\widehat f(v)$ and $\widehat f(w)$ must be in different components of $X \setminus \{z\}$; moreover, they must be vertices in $V$.  But $v$ and $w$ are the nearest vertices to $z$ in those two components.  Hence the path from $f(v)$ to $f(w)$ in the combinatorial tree $(V,E)$ is a path that contains $e$, as desired.
\end{proof}


\section{Infinite Settings}
\label{sc infinite trees}
We can extend both the Tree KKM theorem and the Tree Fixed Point Theorem
to the infinite setting (for covers with infinitely many sets,
and to 
compact maps on trees with infinitely many vertices).

The following standard lemma will be useful, so we include it here for completeness:
\begin{lemma}
\label{le closedintersection}
Let $X$ be a topological space.   Let $\Lambda$ be an
infinite (not necessarily countable) index set and suppose that
$\D=\{D_\alpha : \alpha\in\Lambda\}$ is a family of closed sets in $X$
with the finite intersection property.   Further suppose that
$D_\alpha$ is compact for at least one $\alpha \in \Lambda$.   Then 
$$\bigcap_{\alpha\in \Lambda} \textstyle D_\alpha
\neq \emptyset.$$
\end{lemma} 

\begin{proof}
If not, then there is no point in every $D_\alpha$, hence each point is covered by some complement: $D^c_\alpha$.
Then $\D^c = \{D_\alpha^c : D_\alpha \in \D\}$ is an open cover of
 $Y = \cup \{D_\alpha : \alpha \in \Lambda\}$.
Let $D_{\beta}$ be one of the compact sets in $\D$.   It follows that
$\D^c$ is an open cover of $D_\beta$ and thus has a finite subcover, say it is the collection $\F$.
Since these sets cover $D_\beta$, the intersection of their
complements together with $D_\beta$ is empty --- that is,
$$ \left( \bigcap_{D_\alpha^c \in \F} D_{\alpha}  \right) \cap D_\beta  = \emptyset.$$
This contradicts that $\D$ has the finite intersection property and
thus proves the lemma.
\end{proof}

Now we may prove a KKM theorem for infinite trees:

\begin{theorem}[KKM Theorem for Infinite Trees]
\label{thm:tree-kkm-infinite}
Let $T=(V,E,X)$ be a tree, let $A$ be a (possibly infinite) subset of $V$.
Suppose that $\D=\{D_a : a \in A\}$ is a KKM cover of $T$ relative to $A$ 
such that at least one set in $\D$ is compact.
Then 
$$\bigcap_{a \in A} D_a \neq \emptyset.$$
\end{theorem}

The extra condition that one set be compact is essential.  
Recall again the integer tree: a line whose nodes are the integers, and whose edges are the intervals between successive integers.  See Figure \ref{fig-integer-tree}.  
One may construct a KKM cover by letting $D_i = [i,\infty)$, but the intersection of all such sets is empty.

\begin{figure}[htb]
  \begin{center}
    \scalebox{.6}{\includegraphics{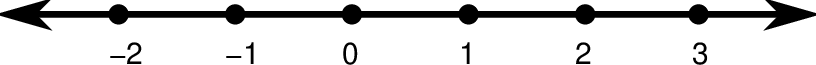}}
  \end{center}
  \caption{The integer tree, with infinite number of vertices (at integers) and edges.}
  \label{fig-integer-tree}
\end{figure}

\begin{proof} 
  By Lemma \ref{le closedintersection} it suffices to show
  that $\D$ has the finite intersection property.
  Let $J$ be a finite subset of $A$.  Then the family of sets $\{D_j : j\in J\}$
  forms a KKM cover of $T$ relative to $J$, a finite set, so 
  Theorem \ref{thm:tree-kkm} applies. Hence $\cap_{\alpha\in J} D_\alpha \neq \emptyset$.
\end{proof}

\begin{figure}[htb]
  \begin{center}
    \scalebox{.6}{\includegraphics{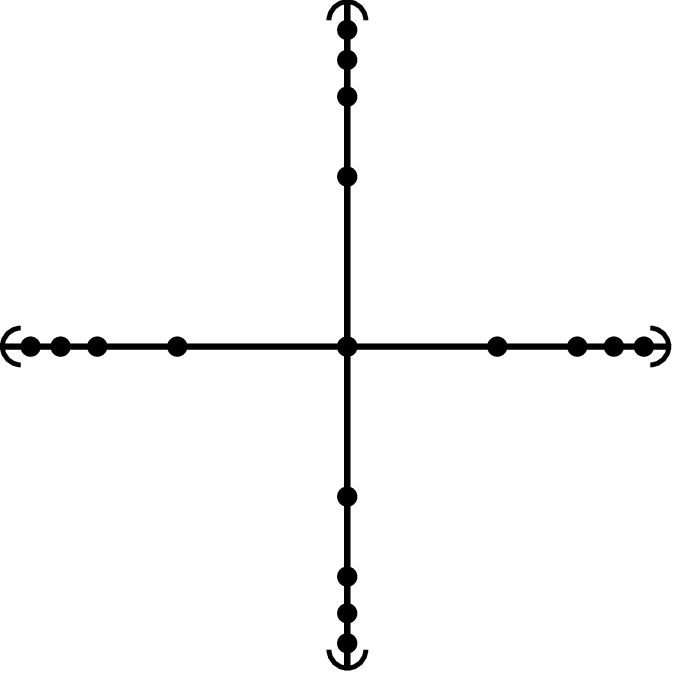}}
  \end{center}
  \caption{A bounded infinite tree.  
      Vertices (not all shown) accumulate at the open endpoints of the underlying set.}
  \label{fig-infinite-tree}
\end{figure}

Using the Tree KKM Theorem we can strengthen the Tree Fixed Point Theorem.
Recall that if $X$ and $Y$ are metric spaces,
$f:X\rightarrow Y$ is a \emph{compact} map 
if $f(B)$ is contained in a compact subset of $Y$ for every bounded set $B$.

\begin{theorem} 
\label{thm:tree-fixed-infinite}
Suppose that $T$ is a bounded tree with vertex set $V$.   If
  $f:T\rightarrow T$ is a continuous compact map, then $f$ has a fixed point. 
\end{theorem}

The compactness hypothesis of the map is reminiscent of the compactness hypothesis of
the Schauder fixed point theorem \cite{DuGr82}.

We give a couple of (non-)examples to illustrate.  Consider the integer tree of Figure \ref{fig-integer-tree} with the map that shifts it one unit to the right.  This map does not have a fixed point;  it is a compact map, but the tree is not bounded.  
The infinite tree of Figure \ref{fig-infinite-tree} is bounded, but it is not compact.  One may easily construct a fixed-point-free map that moves points within this tree towards the rightmost open endpoint, so that the image is not contained in any compact subset.

\begin{proof}
Since $f$ is compact and continuous, 
the image $K=\overline{f(T)}$ is compact and connected.  
Let $V'=V \cap K$ be the set of vertices of $T$ that are in $K$.
Let $\partial K$ denote the set of points that form the boundary of $K$ in $T$.
Consider $V^* = V' \cup \partial K$.  It is a subset of $K$ since $K$ is compact.

We construct a new tree $T^*$ with vertices $V^*$ and underlying metric space $K$.
To start, let $T'$ be the subtree of $T$ spanned by the vertices in $V'$, with edge set $E'$.
We shall augment $T'$ to form $T^*$ by adding points of $\partial K$ as leaves.
So consider any $x \in \partial K$.  Note that $x$ lies in some realized edge $e$ of $T$ with endpoints in $V$.

We claim that exactly one component $C$ of $e \setminus \{x\}$ will intersect $f(T)$.
This is clear if $x$ is an endpoint of $e$, else it follows from the fact that 
$f(T)$ is connected: if $s,t \in f(T)$ were in different components, 
then the path from $s$ to $t$ would lie in $f(T)$ and must contain $x$, so 
$x$ could not be a boundary point of $K$.

Let $v_{x}$ denote the unique endpoint of $e$ that is in component $C$. 
Then either $v_x$ is in $f(T)$ or not. We explore these cases and the tree $T^*$ that results:

\begin{enumerate}
\item
If $v_x \notin f(T)$, then because $f(T)$ is connected, then $f(T)$ must lie in the interior of $e$, so $f(T)$ is an interval with $x$ at one endpoint 
and $y \neq v_x$ at another endpoint.
Then $K$ has at most two boundary points 
and is an interval.
Then let $T^*$ be the tree consisting of one edge $e^*=[x,v_x]$ and two vertices $x,v_{x}$.

\item
If $v_{x} \in f(T)$, then $f(T)$ contains $C$ because $x \in K=\overline{f(T)}$.  
Then the segment $e^*_x=[x,v_x]$ lies in $K$, and $v_x \in V'=V\cap K$.
This construction may be done for every $x \in \partial K$. 
We construct $T^*$ as the tree with vertex set 
$V^*=V' \cup \{x : x \in \partial K \setminus V \}$,
edge set $E^* = E' \cup \{e^*_x : x \in \partial K \setminus V \}$, and underlying space $K$.
\end{enumerate}

Since $T^*$ has underlying space $K$, note that 
$f:T \rightarrow T$ restricts to a function $f^*:T^* \rightarrow T^*$.
Then consider $\D_{f^*,V^*}$, the move-away KKM cover of $T^*$ relative to $V^*$.

Suppose that $H$ is a finite subset of $V^*$.
From Theorem \ref{thm:fixedpt-kkmcover} we see that 
$\D_{f^*,H}$ is a KKM cover of $T$ relative to $H$ and thus,
by the Tree KKM Theorem, the intersection of its sets is nonempty.
Thus, $\D_{f^*,V^*}$ has the finite intersection property.

Because $K$ is compact, all the sets of $\D_{f^*,V^*}$ are compact,
so Theorem \ref{thm:tree-kkm-infinite} shows that the 
intersection of sets in $\D_{f,V}$ is also non-empty, say it contains a point $z$.
Then for all $v \in V^*$ we have $d(z,v) \leq d(f(z),v)$.
We claim that $f(z)=z$.


  Suppose that $z\neq f(z)$.   If $f(z)$ were a vertex $v$, this would contradict that $z\in D_v$.
  So $f(z)$ is not a vertex and is thus a cut point.   
  Choose $w \in V^{*}$ such that $w$ and $z$ are
  in different components of $T\setminus \{f(z)\}$.   
  Then the path from $w$ to $f(z)$ is 
  contained in the path from $w$ to $z$.   Since
  $f(z)\neq z$, this implies that $d(z,w) > d(f(z),w)$, so that $z
  \notin D_w$, a contradiction.   Hence $f(z)=z$.

\end{proof}

\section{A KKM Theorem for Cycles}
\label{sc cycles}
Recall that a {\em cycle} is a finite graph with vertices
$v_1,\dots,v_n$ and edges $(v_i,v_{i+1})$ as well as $(v_n,v_1)$.  
We define a {\em metric cycle} to be a triple $C=(V,E,X)$ where $(V,E)$ is a cycle (as above) and
$X$ is an {\em underlying metric space} obtained from a cycle in exactly the same fashion 
as we obtained a metric tree from a tree in Section \ref{sec:metric-trees}.  The resulting space $C$ is
topologically a circle, partitioned into a finite set of segments (realized edges) joined at their endpoints $v_1,...,v_n$.
Between any two points of $C$ there are exactly two paths; 
the metric is just the minimum length of the two paths connecting the points.

Hereafter, all metric cycles will simply be referred to as cycles.

If $C_n$ is a cycle with $n$ vertices and $x$ is in $C_n$, let $e(x)$
be the set consisting of $x$ and all points $y$ that are not vertices but are on a realized edge with $x$.
Note that by removing from $C_n$ the set $e(x)$ as well as vertices and realized edges within, we obtain a metric tree
$C_n\setminus e(x)$ with $n$ or $(n-1)$
vertices, depending on whether or not $x$ is a vertex of $C_n$.
This observation will become the key
to reducing KKM covers on cycles to KKM covers on trees.  

Now, since there are two paths connecting any two distinct
vertices in a cycle, we must slightly alter our definition of KKM cover for trees, but we
want to do so in a way consistent with our definition for trees.

With this in mind, let $C_n$ be a cycle
with $n$ vertices $V=\{v_1,v_2,\dots , v_n\}$.   A {\em KKM cover} of the cycle $C_n$ is a
family of closed sets $\D = \{D_v : v \in V \}$ such that the
following conditions hold:
\begin{itemize}
\item each $v \in D_v$, and
\item for all $v, w \in V$, {\em at least one of the paths} between $v$ and $w$ is contained
  in $D_v \cup D_w$.
\end{itemize}
This new path-covering property generalizes the corresponding property for trees.
We can now state the main theorem of this section.

\begin{theorem}  
\label{thm:cycle-kkm}
Suppose that $C_n$ is a cycle with vertices $V=\{v_1,v_2,\dots, v_n\}$ and let 
  $\D = \{D_v:
  v \in V \}$ be a KKM cover of the cycle $C_n$.   Then there is a point
  $x$ in $C_n$ such that $x$ is in at least $\lfloor\frac{n}{2}\rfloor
  + 1$ sets of $\D$.
\end{theorem}

To compare this result with Theorem \ref{thm:tree-kkm}, note that KKM covers of trees have a point in {\em all} the sets of the cover, but 
KKM covers of cycles have a point in a {\em strict majority} of the sets.

\begin{proof}
For each $x$, consider the set $\ell(x) = \{v \in V : x \in D_v\}$.   

Fix $x$.  If $|\ell(x)| \geq \flr{n}{2} +1$, then we have our desired conclusion.

Else, if $|\ell(x)| \leq   \flr{n}{2} -1$, then 
  let $H= V \setminus \ell(x)$; clearly $|H| \geq \flr{n}{2}+1$.  
  Note that for any pair $v$ and $w$ in $H$, the path between them covered by $D_v\cup D_w$ does not include $x$, 
  so this path is still covered if we remove $e(x)$ from the graph.   Then
  $C_n\setminus e(x)$ is a tree and and the family $\mathscr{F} = \{D_v: v \in H\}$ is a KKM cover of this tree relative to $H$.   
  By Theorem \ref{thm:tree-kkm}, $\mathscr{F}$
  has non-empty intersection, and it has at least $\flr{n}{2}+1$ sets of $\D$, as desired.

The only remaining case is when $|\ell(x)| = \flr{n}{2}$ is constant for all $x$.
We now show why this leads to a contradiction, but we must take some care because the sets of the cover might have several connected components.

The set of boundary points of $D_v$, denoted by $\partial D_v$, is a closed set.
Moreover, 
$\partial D_v$ has no interior,
so the finite union of boundary points $B = \cup \{ \partial D_v : v \in V\}$ is closed and has no interior. 
So $B^c$ is a nonempty open set.

So choose $x \in B^c$ and since $B$ is closed, we may find $b \in B$ that is closest to $x$.
Then all points in $U = \{ y : d(x,y) < d(x,b) \}$ are, for each $v \in V$, interior points of either $D_v$ or $D_v^c$.  
So if $y \in U$, $\ell(x)=\ell(y)$.  Since $b$ is a limit point of $U$ and each $D_v$ is closed, $\ell(x) \subseteq \ell(b)$.

If for some $v \notin \ell(x)$ we have $b \in D_v$, then $\ell(x) \cup \{v\} \subseteq \ell(b)$, so the size of $\ell$ is not constant, a contradiction.
Otherwise, for every $v \notin \ell(x)$, we have that $b \in D_v^c$, an open set, so $\ell(x) = \ell(b)$.
So there is an open set $W$ around $b$ that contains no points of $D_v$ for all $v \notin \ell(x)$; therefore for $w \in W$, 
$\ell(w) \subseteq \ell(x)=\ell(b)$.
Since $b \in B$, it must be in $\partial D_z$ for some $z \in \ell(x)$.
So there is a $w \in W$ such that $w \notin D_z$, thus $\ell(w) \subseteq \ell(b) \setminus \{z\}$, so the size of $\ell$ is not constant, a contradiction.
\end{proof}

Theorem \ref{thm:cycle-kkm} has an interesting application to voting
theory.  In {\em approval voting}, each voter specifies which options she 
would consider acceptable, without ranking the options.  Following
\cite{BNST}, the set of all options available to voters 
is called a {\em (political) spectrum}; it often has a natural
topology given by notions of ``closeness'' or ``similarity'' 
of political preferences.  For instance, the political spectrum is
often modeled as $\R$, a line with conservative positions to the right and
liberal positions to the left.  However, in elections over multiple
issues, the spectrum might be best modeled as a subset of $\R^n$.
Political spectra have been modeled also by a circle; often this
arises by bending the linear political spectrum so that the extreme
left-wing and right-wing positions are considered close; 
e.g., see \cite{norway-stuff}.

For each voter, the set of options that a voter approves is called her
{\em approval set}.  We assume that approval sets are closed subsets
of the spectrum, and we call the set of all voters together with their
approval sets a {\em society}.

We call a society with a circular political spectrum
{\em super-agreeable} if for each pair of voters
$i,j$, one of the paths between $i,j$ is covered by their approval
sets.  We remark that in many cases it is natural to assume that a voter's
approval set is connected.  In this situation, a super-agreeable
society is simply one in which every pair of voters can find common
ground, i.e., an option which they will both approve.  This agrees
with the definition of super-agreeable for a linear society, as in
\cite{BNST}.  

Then Theorem \ref{thm:cycle-kkm} then has the following corollary:

\begin{theorem}
In a super-agreeable society with a circular political spectrum, there
is an option that will be approved by a strict majority of the voters.
\end{theorem}

The value of this result is that it gives a sufficient condition for
the existence of a strict majority using approval voting when the
political spectrum is circular.  We do not assume the approval sets have to be connected.
Results for connected approval sets and weaker intersection hypotheses may be found in \cite{BNST}, who consider linear political spectra,
and Hardin \cite{Hard}, who extends those results to circular political spectra.

\bibliographystyle{plain}	
\bibliography{kkmtree}
\end{document}